\documentclass[12pt, reqno]{amsart}

\author[P.~Leonetti]{Paolo Leonetti}
\address{Universit\`a ``Luigi Bocconi''\\Department of Statistics\\Milan, Italy}
\email{leonetti.paolo@gmail.com}
\urladdr{\url{https://sites.google.com/site/leonettipaolo/}}


\usepackage{stmaryrd}
\usepackage{amsmath}
\usepackage{amssymb}
\usepackage{amsthm}
\usepackage[left=3.5cm, right=3.5cm, paperheight=11.8in]{geometry}
\usepackage{hyperref}
\usepackage{fancyhdr}
\usepackage{enumitem}
\usepackage{comment}
\usepackage{nicefrac}
\usepackage{mathrsfs}
\usepackage{url}
\usepackage{graphicx}
\usepackage[utf8]{inputenc}
\usepackage{cancel}

\AtBeginDocument{%
   \def\MR#1{}
}

\theoremstyle{theorem}
\newtheorem{thm}{Theorem}

\newtheorem{cor}[thm]{Corollary}
\newtheorem{question}[thm]{Question}
\theoremstyle{question}

\theoremstyle{definition}

\hypersetup{
    pdftitle={Convergence Rates of Subseries},
    pdfauthor={},
    pdfmenubar=false,
    pdffitwindow=true,
    pdfstartview=FitH,
    colorlinks=true,
    linkcolor=blue,
    citecolor=green,
    urlcolor=cyan
}

\providecommand{\MR}[1]{}

\providecommand{\MR}{\relax\ifhmode\unskip\space\fi MR }

\providecommand{\href}[2]{#2}

\begin{document}

\title{Convergence Rates of Subseries}

\maketitle
\thispagestyle{empty}

\begin{abstract} 
Let $(x_n)$ be a positive real sequence decreasing to $0$ such that the series $\sum_n x_n$ is divergent and $\liminf_{n} x_{n+1}/x_n>1/2$. 
We show that there exists a constant $\theta \in (0,1)$ such that, for each $\ell>0$, there is a subsequence $(x_{n_k})$ for which $\sum_k x_{n_k}=\ell$ and $x_{n_k}=O(\theta^k)$.
\end{abstract}

\section{Introduction.} 

Given a real sequence $x=(x_n)$ with divergent series $\sum_n x_n$, let $\mathcal{A}(x)$ be its \emph{achievement set}, that is, the set of sums of convergent subseries 
$$
\textstyle \mathcal{A}(x)=\left\{\ell \in \mathbf{R}: \sum_{k\ge 1} x_{n_k}=\ell \text{ for some subsequence }(x_{n_k})\right\}.
$$
In this regard, it is known that ``almost all'' subseries (both in the measure-theoretic sense and the categorical sense) diverge, see \cite{MR0009997, MR0493027, MR0179507}. Related results in the context of filter convergence have been recently proved in \cite{BPW2018, Leo17, LMM}.
 
There is a large literature on achievement sets; see, e.g., \cite{MR978930, MR0005391, MR2812282, MR3418208, MR1750775}. 
Achievement sets are also known as \emph{subsum sets}: it is known, for instance, that if $(x_n)$ is a sequence whose terms form a conditionally convergent series then $\mathcal{A}(x)=\mathbf{R}$; see \cite[Corollary 1.4]{MR2812282}. 
In addition, if $(x_n)$ is a sequence converging to $0$ then $\mathcal{A}(x)$ is one of the following three possibilities: a finite union of (nontrivial) compact intervals, a Cantor set, or a ``symmetric Cantorval'' (that is, a Cantor-like set with both trivial and nontrivial components); see \cite[p. 870]{MR3418208}. 
Motivations for the investigation of these sets come from, among other areas, measure theory and fractal geometry; see \cite{MR745282} and \cite{MR1186131}, respectively. Of special interest have been specific subseries 
of the harmonic series $\sum_n \frac{1}{n}$; see, e.g., \cite{MR528792,  MR0289994, MR2416253, MR0380183, MR508230}.

Motivated by the study of achievement sets in the context of Banach spaces \cite{MR3730451}, Jacek Marchwicki posed the following question during the open problem session of the 46th Winter School in Abstract Analysis (Czech Republic, 2018):
\begin{question}\label{question1}
Fix a real $\ell>0$. Does there exist a set of distinct primes $\mathcal{P}$ such that
$$
\sum_{p \in \mathcal{P}} \frac{1}{p} = \ell \,\,\,\text{ and }\,\,\,\sum_{p \in \mathcal{P}} \frac{1}{\sqrt{p}}<\infty\,?
$$
\end{question}
\noindent Informally, the question above aims to establish whether there is a relationship between numbers $\ell$ in the achievement set $\mathcal{A}(x)$, where each $x_n$ is the reciprocal of the $n$th prime, and the growth rate of the subsequences $(x_{n_k})$ for which $\sum_{k} x_{n_k}=\ell$.

In this respect, we are going to prove that there exists a sequence $(x_{n_k})$ of this type that decays at least exponentially. 
As a consequence of our main result (Theorem \ref{th:main} below), we give an affirmative answer to Question \ref{question1}, and more, by proving the following corollary to it.
\begin{cor}\label{cor0}
For each $\ell>0$ and for each $\theta>0$, there exists an increasing sequence of primes $(p_k)$ and a constant $c>0$ such that 
$$\sum_k \frac{1}{p_k}=\ell\,\,\,\,\text{ and }\,\,\,\,p_k \ge c\,\theta^k\,\,\text{ for all }\,k.$$ 
In particular, $\sum_k\frac{1}{p_k^\alpha}<\infty$ for all $\alpha>0$.
\end{cor}

A related result on the rate of growth of the partial sums $\left(\sum_{i=1}^n x_i\right)$ can be found in \cite{MR692583}. 
Hereafter, given a real sequence $(a_n)$ and a function $f: \mathbf{N} \to (0,\infty)$, the notation $a_n=O(f(n))$ stands for $\limsup_{n\to \infty}a_n/f(n)<\infty$. 
The main result of this note is the following.
\begin{thm}\label{th:main}
Let $(x_n)$ be a positive nonincreasing real sequence such that $\sum_n x_n=+\infty$, $\lim_{n\to \infty} x_n = 0$, and
\begin{equation}\label{eq:limsup}
L:=\liminf_{n\to \infty} \frac{x_{n+1}}{x_n}>\frac{1}{2}.
\end{equation}
Then there exists a constant $\theta \in (0,1)$ such that, for each $\ell>0$, there is a subsequence $(x_{n_k})$ for which $\sum_{k} x_{n_k}=\ell$ and $x_{n_k}=O(\theta^k)$.
\end{thm}
It is worth noting that it is well known that every sum can be achieved, i.e., $\mathcal{A}(x)=(0,\infty)$; see \cite[Theorem 1.1]{MR0252890} and \cite[p. 863]{MR3418208}. Hence the novelty here concerns the rate of convergence of subseries, which is asymptotically (at most) geometric. In addition, these rates are bounded uniformly for all desired subsums. 

As it follows from the proof of Theorem \ref{th:main}, an estimate on the value of $\theta$ can be found if $L$ is not small:
\begin{cor}\label{corollary1}
With the same hypotheses as in Theorem \ref{th:main}, suppose that
$$
L:=\liminf_{n\to \infty} \frac{x_{n+1}}{x_n}>\frac{\sqrt{5}-1}{2} \,\,(\approx 0.618034).
$$
Then, for each $\ell>0$ and for each $\varepsilon>0$, there exists a subsequence $(x_{n_k})$ for which $\sum_{k} x_{n_k}=\ell$ and $x_{n_k}=O\left(\sqrt{1+\varepsilon-L}^{\,k}\right)$.
\end{cor}

In particular, if $L=1$, we obtain the following result (we omit details).
\begin{cor}\label{cor5}
Let $(x_n)$ be a positive nonincreasing real sequence such that $\sum_n x_n=+\infty$, $\lim_{n\to \infty} x_n = 0$, and $\lim_{n\to \infty} \frac{x_{n+1}}{x_n}=1$. Then, for each $\ell>0$ and for each $\theta>0$, there exists a subsequence $(x_{n_k})$ for which $\sum_{k} x_{n_k}=\ell$ and $x_{n_k}=O(\theta^{k})$.
\end{cor}

Finally, we have the following easy consequence.
\begin{cor}\label{corollary2}
With the hypotheses of Theorem \ref{th:main}, for each $\ell>0$ there exists a subsequence $(x_{n_k})$ such that $\sum_{k} x_{n_k}=\ell$ and $\sum_{k}x_{n_k}^\alpha <\infty$ for all $\alpha>0$.
\end{cor}

\begin{proof}
If $\alpha<1$, the claim follows easily by Theorem \ref{th:main} (we omit details). Hence, let us suppose that $\alpha \ge 1$. Since $x$ is convergent to $0$, there exists an integer $N\ge 1$ such that $x_n<1$ for all $n\ge N$. Therefore
$$
\sum_{k\ge 1}x_{n_k}^\alpha \le N\,\|x\|_\infty^\alpha +\sum_{k\ge N}x_{n_k} \le O(1)+\ell < \infty. \eqno\qedhere
$$
\end{proof}

\begin{proof}[Proof of Corollary \ref{cor0}] 
Thanks to the prime number theorem, the $n$th prime is equal to $n\log n(1+o(1))$, so that the inequality \eqref{eq:limsup} holds with $L=1$. The result follows by Corollary \ref{cor5} and 
Corollary \ref{corollary2}.
\end{proof}


\section{Further Proofs.}

\begin{proof}[Proof of Theorem \ref{th:main}] First, suppose that $\alpha \in (0,1)$ and fix $\ell>0$. Let $a_1$ be the smallest integer $a$ such that $x_{a}<\ell$, which is well-defined because $\lim_{n\to \infty} x_n=0$. Also, let $b_1$ be the greatest integer $b\ge a_1$ such that $\sum_{n=a_1}^b x_n< \ell$. Define $S_1:=\sum_{j=a_1}^{b_1}x_j$. Then recursively construct the increasing sequences of integers $(a_n)$ and $(b_n)$, as well as the related sequence of real numbers $S_n:=\sum_{j=a_n}^{b_n}x_j$, in the analogous way:

\medskip

(i) Let $a_{n+1}$ be the smallest integer $a>b_n$ such that 
$$
x_a < \ell -\sum_{i=1}^n S_i;
$$

\medskip

(ii) Let $b_{n+1}$ be the greatest integer $b \ge a_{n+1}$ such that 
$$
\sum_{j=a_{n+1}}^b x_j < \ell-\sum_{i=1}^{n} S_i.
$$

\medskip

\noindent Note that, by the maximality of $b_n$, we have $a_{n+1} \ge b_n+2$. Denoting by $(n_k)$ the increasing enumeration of (infinite) set $\bigcup_{n} [a_n,b_n]$, it easily follows that $\sum_{k} x_{n_k}=\ell$. Indeed,  
$$
0\le \ell-\sum_{i=1}^{n} S_i \le x_{b_n+1}\to 0
$$
as $n\to \infty$. 

At this point, 
fix $\varepsilon>0$ sufficiently small so that $\tilde{L}:=L-\varepsilon>\frac{1}{2}$, and set $N=N(\varepsilon) > 1$ such that
$$
x_{n+1} \ge \tilde{L}\, x_n
$$ 
for all $n\ge N$. 
Note that 
$
\sum_{i=1}^n S_i <\ell \le \sum_{i=1}^{n-1}S_i+x_{a_{n}-1}
$ 
for all $n>1$, implying that
\begin{equation}\label{eq:mainineq}
x_{a_n-1}>S_n=x_{a_n}+\cdots+x_{b_n} \ge \left(\tilde{L}+\cdots+\tilde{L}^{\kappa_n}\right)x_{a_n-1}
\end{equation}
for all $n\ge N$, where $\kappa_n:=b_n-a_n+1$. 
Considering that $\tilde{L}>\frac{1}{2}$, it easily follows that 
\begin{equation}\label{eq:limsupbound}
K:=\limsup_{n\to \infty}\kappa_n < \infty.
\end{equation}

With similar reasoning, we have $\sum_{i=1}^{n+1} S_i <\ell \le \sum_{i=1}^{n-1}S_i+x_{a_{n}-1}$ for all $n>1$. This implies that
\begin{displaymath}
\begin{split}
\kappa_{n+1}x_{b_{n+1}} \le S_{n+1}&<x_{a_n-1}-S_n \le x_{a_n-1}-x_{a_n}  \\
&\le x_{a_n-1}(1-\tilde{L}) \le x_{b_{n-1}}(1-\tilde{L}).
\end{split}
\end{displaymath}
Therefore, 
$
x_{b_n}=O\left((1-\tilde{L})^{n/2}\right)
$. 
Considering that there exists an integer $M\le 0$ such that $n_k \ge b_{\lfloor k/K\rfloor}+M$ for all large $k$, we conclude that 
\begin{equation}\label{eq:claimcorollary}
x_{n_k} \le x_{b_{\lfloor k/K\rfloor}+M}\le x_{b_{\lfloor k/K\rfloor+M}} =O\left( (1-\tilde{L})^{\,\frac{k/K+M}{2}}\right)=O(\theta^k),
\end{equation}
where $\theta:=(1-\tilde{L})^{1/2K}$. This completes the proof.
\end{proof}

\begin{proof}[Proof of Corollary \ref{corollary1}] With the same notation as in the proof above, fix a sufficiently small $\varepsilon>0$ such that $\hat{L}:=L-\varepsilon>\frac{\sqrt{5}-1}{2}$. Then inequality \eqref{eq:mainineq} holds only if $\kappa_n=1$ for all $n\ge 2$; hence $K=1$. The claim follows by the estimate \eqref{eq:claimcorollary}.
\end{proof}

\subsection{Acknowledgment}
The author is grateful to the editor and the two anonymous referees for their remarks that allowed a substantial improvement of the presentation.


\bibliographystyle{amsplain}
\bibliography{serie}

\end{document}